\documentclass{amsart}
\usepackage{amssymb}
\usepackage{amsmath}
\usepackage{amsfonts}

\newtheorem{theorem}{Theorem}
\theoremstyle{plain}

\newtheorem{definition}{Definition}
\newtheorem{example}{Example}

\newtheorem{lemma}{Lemma}

\newtheorem{proposition}{Proposition}
\newtheorem{remark}{Remark}

\numberwithin{equation}{section}
\input{tcilatex}

\begin{document}
\title{CURVATURE PROPERTIES OF QUASI-PARA-SASAKIAN MANIFOLDS}
\author{I. K\"{u}peli Erken}
\address{Faculty of Engineering and Natural Sciences, Department of
Mathematics, Bursa Technical University, Bursa, TURKEY}
\email{irem.erken@btu.edu.tr}
\date{05.03.2018}

\begin{abstract}
The present paper is devoted to quasi-Para-Sasakian manifolds. Basic
properties of such manifolds are obtained and general curvature identities
are investigated. Next it is proved that if $M$ is quasi-Para-Sasakian
manifold of constant curvature $K$. Then $K$ $\leq 0$ and $(i)~$if $K=0$,
the manifold is paracosymplectic, $(ii)$ if $K<0$, the quasi-para-Sasakian
structure of $M$ is obtained by a homothetic deformation of a para-Sasakian
structure.
\end{abstract}

\maketitle


\section{I\textbf{ntroduction}}

Almost paracontact metric structures are the natural odd-dimensional
analogue to almost paraHermitian structures, just like almost contact metric
structures correspond to the almost Hermitian ones. The study of almost
paracontact geometry was introduced by Kaneyuki and Williams in \cite%
{kaneyuki1} and then it was continued by many other authors. A systematic
study of almost paracontact metric manifolds was carried out in paper of
Zamkovoy, \cite{Za}. Comparing with the huge literature in almost contact
geometry, it seems that there are necessary new studies in almost
paracontact geometry. Therefore, paracontact metric manifolds have been
studied in recent years by many authors, emphasizing similarities and
differences with respect to the most well known contact case. Interesting
papers connecting these fields are (see, for example, \cite{DACKO}, \cite%
{Biz}, \cite{Welyczko}, \cite{Za}, and references therein).

Z. Olszak studied normal almost contact metric manifolds of dimension $3$ 
\cite{olszak}. He derive certain necessary and sufficient conditions for an
almost contact metric structure on manifold to be normal and curvature
properties of such structures and normal almost contact metric structures on
a manifold of constant curvature are studied. Recently, J. We\l yczko
studied curvature and torsion of Frenet-Legendre curves in $3$-dimensional
normal almost paracontact metric manifolds \cite{Welyczko1} and then normal
almost paracontact metric manifolds are studied by \cite{bejan}, \cite{irem1}%
, \cite{irem2}.

The notion of quasi-Sasakian manifolds, introduced by D. E. Blair in \cite%
{BL1}, unifies Sasakian and cosymplectic manifolds. By definition, a
quasi-Sasakian manifold is a normal almost contact metric manifold whose
fundamental $2$-form $\Phi :=g(\cdot ,\phi \cdot )$ is closed.
Quasi-Sasakian manifolds can be viewed as an odd-dimensional counterpart of
Kaehler structures. These manifolds studied by several authors(e.g. \cite%
{KANEMAKI}, \cite{OL1}, \cite{TAN1}).

Although quasi-Sasakian manifolds were studied by several different authors
and are considered a well-established topic in contact Riemannian geometry,
there do not exist any study about quasi-Para-Sasakian manifolds as far as
the author knows.

Motivated by these considerations, in this paper we make the first
contribution to investigate basic properties and general curvature
identities of quasi-Para-Sasakian manifolds.

The paper is organized in the following way.

Section $2$ is preliminary section, where we recall the definition of almost
paracontact metric manifold and quasi-Para-Sasakian manifolds.

In Section $3$, we study basic properties and curvature identities of such
manifolds.

In the short auxiliary Section $4$, \emph{D-}homothetic deformations of
quasi-Para-Sasakian structures are studied and in Section $5$, we
characterize quasi-Para-Sasakian manifolds of constant curvature. Finally,
an example is given.

\section{Preliminaries}

Let $M$ be a $(2n+1)$-dimensional differentiable manifold and $\phi $ is a $%
(1,1)$ tensor field, $\xi $ is a vector field and $\eta $ is a one-form on $%
M.$ Then $(\phi ,\xi ,\eta )$ is called an \textit{almost paracontact
structure} on $M$ if

\begin{itemize}
\item[(i)] $\phi^2 = Id-\eta\otimes\xi, \quad \eta(\xi)=1$,

\item[(ii)] the tensor field $\phi $ induces an almost paracomplex structure
on the distribution $D=$ ker $\eta ,$ that is the eigendistributions $D^{\pm
},$ corresponding to the eigenvalues $\pm 1$, have equal dimensions, $\text{%
dim}\,D^{+}=\text{dim}\,D^{-}=n$.
\end{itemize}

The manifold $M$ is said to be an \textit{almost paracontact manifold} if it
is endowed with an almost paracontact structure \cite{Za}.

Let $M$ be an almost paracontact manifold. $M$ will be called an \textit{%
almos}t \textit{paracontact metric manifold} if it is additionally endowed
with a pseudo-Riemannian metric $g$ of a signature $(n+1,n)$, i.e.%
\begin{equation}
g(\phi X,\phi Y)=-g(X,Y)+\eta (X)\eta (Y).  \label{1}
\end{equation}

For such manifold, we have 
\begin{equation}
\eta (X)=g(X,\xi ),\text{ }\phi (\xi )=0,\text{ }\eta \circ \phi =0.
\label{2}
\end{equation}

Moreover, we can define a skew-symmetric tensor field (a $2$-form) $\Phi $ by%
\begin{equation}
\Phi (X,Y)=g(X,\phi Y),  \label{3}
\end{equation}%
usually called \textit{fundamental form}.

For an almost paracontact manifold, there exists an orthogonal basis $%
\{X_{1},\ldots ,X_{n},Y_{1},\ldots ,$ $Y_{n},\xi \}$ such that $%
g(X_{i},X_{j})=\delta _{ij}$, $g(Y_{i},Y_{j})=-\delta _{ij}$ and $Y_{i}=\phi
X_{i}$, for any $i,j\in \left\{ 1,\ldots ,n\right\} $. Such basis is called
a $\phi $\textit{-basis}.

On an almost paracontact manifold, one defines the $(1,2)$-tensor field $%
N^{(1)}$ by%
\begin{equation}
N^{(1)}(X,Y)=\left[ \phi ,\phi \right] (X,Y)-2d\eta (X,Y)\xi ,  \label{nijen}
\end{equation}%
where $\left[ \phi ,\phi \right] $ is the \textit{Nijenhuis torsion} of $%
\phi $%
\begin{equation*}
\left[ \phi ,\phi \right] (X,Y)=\phi ^{2}\left[ X,Y\right] +\left[ \phi
X,\phi Y\right] -\phi \left[ \phi X,Y\right] -\phi \left[ X,\phi Y\right] .
\end{equation*}

If $N^{(1)}$ vanishes identically, then the almost paracontact manifold
(structure) is said to be \textit{normal} \cite{Za}. The normality condition
says that the almost paracomplex structure $J$ defined on $M\times 
\mathbb{R}
$%
\begin{equation*}
J(X,\lambda \frac{d}{dt})=(\phi X+\lambda \xi ,\eta (X)\frac{d}{dt}),
\end{equation*}%
is integrable.

If $d\eta (X,Y)=$ $g(X,\phi Y)=\Phi (X,Y)$, then $(M,\phi ,\xi ,\eta ,g)$ is
said to be \emph{paracontact metric manifold.} In a paracontact metric
manifold one defines a symmetric, trace-free operator $h=\frac{1}{2}{%
\mathcal{L}}_{\xi }\phi $, where $\mathcal{L}_{\xi }$, denotes the Lie
derivative. It is known \cite{Za} that $h$ anti-commutes with $\phi $ and
satisfies $h\xi =0,$ tr$h=$tr$h\phi =0$ and $\nabla \xi =-\phi +\phi h,$
where $\nabla $ is the Levi-Civita connection of the pseudo-Riemannian
manifold $(M,g)$.

Moreover $h=0$ if and only if $\xi $ is Killing vector field. In this case $%
(M,\phi ,\xi ,\eta ,g)$ is said to be a \emph{K-paracontact manifold}.
Similarly as in the class of almost contact metric manifolds \cite{blair2}$,$
a normal almost paracontact metric manifold will be called \textit{%
para-Sasakian }if $\Phi =d\eta $ \cite{erdem}.

\begin{definition}
\label{mert}An almost paracontact metric manifold $(M^{2n+1},\phi ,\xi ,\eta
,g)$ is called \textit{quasi-para-Sasakian if the structure is normal and
its fundamental }$2$-form $\Phi $ is closed.
\end{definition}

The class of para-Sasakian manifolds is contained in the class of
quasi-para-Sasakian manifolds. The converse does not hold in general. Also
in this context the para-Sasakian condition implies the $K$-paracontact
condition and the converse holds only in dimension $3$. A paracontact metric
manifold will be called\textit{\ paracosymplectic} if $d\Phi =0,$ $d\eta =0$ 
\cite{DACKO}, also, the class of paracosymplectic manifolds is contained in
the class of quasi-para-Sasakian manifolds.

\section{Basic Structure and Curvature Identities}

\begin{definition}
\label{AX} For a \textit{quasi-para-Sasakian manifold} $(M^{2n+1},\phi ,\xi
,\eta ,g)$, define the $(1,1)$ tensor field $\mathcal{A}$ by%
\begin{equation}
\mathcal{A}X=\nabla _{X}\xi .  \label{4}
\end{equation}

\begin{remark}
For the easy readability of the identities, we will use $g(\mathcal{A}%
X,Y)=(\nabla _{X}\eta )Y$.
\end{remark}
\end{definition}

Since the proof of the following Lemma is quite similar to Lemma 4.1 of \cite%
{BL1}, we don't give the proof of it.

\begin{lemma}
\label{properties1} Vector field $\xi $ of a \textit{quasi-para-Sasakian
structure} $(\phi ,\xi ,\eta ,g)$ is a Killing vector field.%
\begin{equation}
g(\mathcal{A}X,Y)+g(X,\mathcal{A}Y)=0.  \label{p1}
\end{equation}
\end{lemma}

\begin{proposition}
\label{properties2}For a \textit{quasi-para-Sasakian manifold }$%
(M^{2n+1},\phi ,\xi ,\eta ,g),$ we have 
\begin{equation}
(\nabla _{X}\phi )Y=-g(\mathcal{A}X,\phi Y)\xi -\eta (Y)\phi \mathcal{A}X,
\label{P5}
\end{equation}%
\begin{equation}
\nabla _{\xi }\phi =0,~~~\nabla _{\xi }\xi =0,~~~\nabla _{\xi }\eta =0,
\label{P6}
\end{equation}%
\begin{equation}
\mathcal{A}\phi X=\phi \mathcal{A}X,  \label{P2}
\end{equation}%
\begin{equation}
g(\mathcal{A}\phi X,\phi Y)=-g(\mathcal{A}X,Y),  \label{P3}
\end{equation}%
\begin{equation}
g(\mathcal{A}\phi X,Y)=-g(\mathcal{A}X,\phi Y),  \label{P4}
\end{equation}%
where $X$, $Y$ are arbitrary vector fields on $M^{2n+1}$.
\end{proposition}

\begin{proof}
Using the Cartan magic formula%
\begin{equation*}
\mathcal{L}_{\xi }\Phi =d(i_{\xi }\Phi )+i_{\xi }(d\Phi ),
\end{equation*}%
we find $\mathcal{L}_{\xi }\Phi =0$, since $d\Phi =0$ and $(i_{\xi }\Phi
)X=\Phi (\xi ,X)=g(\xi ,\phi X)=0$, where $\mathcal{L}$ indicates the
operator of the Lie differentiation. If we use the definition of
quasi-para-Sasakian manifold, (\ref{p1}) and the well known equation $2d\eta
(X,Y)=X(\eta (Y))-Y(\eta (X))-\eta ([X,Y])$ in Proposition 2.4 of \cite{Za},
we obtain (\ref{P5}).

$\mathcal{L}_{\xi }\Phi =0,$ properties of $\phi $ and eq. (\ref{P6}) follow 
\begin{eqnarray*}
(\mathcal{L}_{\xi }\Phi )(X,Y) &=&\mathcal{L}_{\xi }\Phi (X,Y)-\Phi (%
\mathcal{L}_{\xi }X,Y)-\Phi (X,\mathcal{L}_{\xi }Y) \\
0 &=&g(\phi \mathcal{A}Y-\mathcal{A}\phi Y,X).
\end{eqnarray*}%
So we obtain (\ref{P2}). In virtue of (\ref{P2}), we obtain (\ref{P3}) and (%
\ref{P4}).
\end{proof}

\begin{lemma}
\label{2-form}For a \textit{quasi-para-Sasakian} manifold $(M^{2n+1},\phi
,\xi ,\eta ,g)$ with its curvature transformation $R_{XY}=[\nabla
_{X},\nabla _{Y}]-\nabla _{\lbrack X,Y]}$, the following equations hold%
\begin{equation}
R(\xi ,X)Y=-(\nabla _{X}\mathcal{A)}Y,  \label{R1}
\end{equation}%
\begin{equation}
g(R(\xi ,X)Y,\xi )=g(\mathcal{A}X,\mathcal{A}Y),  \label{R 1.1}
\end{equation}%
\begin{equation}
g(R(\xi ,X)\phi Y,\phi Z)+g(R(\xi ,X)Y,Z)=g(\mathcal{A}X,\mathcal{A}Y)\eta
(Z)-g(\mathcal{A}X,\mathcal{A}Z)\eta (Y),  \label{R1.2}
\end{equation}%
\begin{equation}
S(\xi ,\xi )=-tr\mathcal{A}^{2}.  \label{R1.3}
\end{equation}
\end{lemma}

\begin{proof}
Using the fact that $\xi $ is Killing vector field and equation (\ref{p1}),
one can easily get (\ref{R1}). If we take the inner product of (\ref{R1})
with $\xi $ and then use (\ref{p1}), we have (\ref{R 1.1}). Using (\ref{P5}%
), we get%
\begin{eqnarray}
\mathcal{A}\nabla _{X}\phi Y &=&-\eta (Y)\phi \mathcal{A}^{2}X+\phi \mathcal{%
A}\nabla _{X}Y,  \label{A} \\
\mathcal{A(}\nabla _{X}\phi )Y &=&-\eta (Y)\phi \mathcal{A}^{2}X.  \notag
\end{eqnarray}%
From (\ref{P2}) and (\ref{P5}), we have%
\begin{equation}
\phi \nabla _{X}\mathcal{A}Y-\nabla _{X}\mathcal{A}\phi Y=-(\nabla _{X}\phi )%
\mathcal{A}Y=g(\mathcal{A}X,\phi \mathcal{A}Y)\xi .  \label{B}
\end{equation}%
Taking into account (\ref{R1}), (\ref{A}) and (\ref{B}), we obtain%
\begin{equation}
R(\xi ,X)\phi Y-\phi R(\xi ,X)Y=-\eta (Y)\phi \mathcal{A}^{2}X+g(\mathcal{A}%
X,\phi \mathcal{A}Y)\xi .  \label{C}
\end{equation}%
On the other hand, if we take the inner product of (\ref{C}) with $\phi Z$
and use (\ref{p1}) and (\ref{R 1.1}), we get (\ref{R1.2}). The proof of (\ref%
{R1.3}) is a direct consequence of (\ref{R 1.1}).
\end{proof}

\begin{proposition}
\label{also have}For a \textit{quasi-para-Sasakian} manifold $(M^{2n+1},\phi
,\xi ,\eta ,g)$, we also have%
\begin{eqnarray}
g(R(X,Y)\phi Z,\phi W)+g(R(X,Y)Z,W) &=&\eta (W)g(R(X,Y)Z,\xi )+\eta
(Z)g(R(X,Y)\xi ,W)  \notag \\
&&-g(\mathcal{A}X,\phi W)g(\mathcal{A}Y,\phi Z)+g(\mathcal{A}X,\phi Z)g(%
\mathcal{A}Y,\phi W)  \notag \\
&&+g(\mathcal{A}X,Z)g(\mathcal{A}Y,W)-g(\mathcal{A}X,W)g(\mathcal{A}Y,Z).
\label{RXYY}
\end{eqnarray}
\end{proposition}

\begin{proof}
The following formula is valid%
\begin{equation*}
(\nabla _{X}\nabla _{Y}\phi )Z=\nabla _{X}(\nabla _{Y}\phi )Z-(\nabla
_{_{\nabla _{X}Y}}\phi )Z-(\nabla _{Y}\phi )\nabla _{X}Z.
\end{equation*}%
Now we suppose that $P$ is a fixed point of $(M^{2n+1},\phi ,\xi ,\eta ,g)$
and $X,Y,Z$ are vector fields such that $(\nabla X)_{P}=$ $(\nabla
Y)_{P}=(\nabla Z)_{P}=0$ Hence the last identity at the point $P$, reduces
to the form 
\begin{equation}
(\nabla _{X}\nabla _{Y}\phi )Z=\nabla _{X}(\nabla _{Y}\phi )Z-(\nabla
_{Y}\phi )\nabla _{X}Z.  \label{rxy1}
\end{equation}%
Now, after differentiating (\ref{P5})\textit{\ }covariantly and using (\ref%
{rxy1}),\ we find%
\begin{equation*}
(\nabla _{X}\nabla _{Y}\phi )Z=-g(\nabla _{X}\mathcal{A}Y,\phi Z)\xi -g(%
\mathcal{A}Y,\phi Z)\mathcal{A}X-g(\mathcal{A}X,Z)\phi \mathcal{A}Y-\eta
(Z)\phi \nabla _{X}\mathcal{A}Y.
\end{equation*}%
On the other hand, combining the last equation and (\ref{P5}), we obtain%
\begin{eqnarray}
(R(X,Y)\phi )Z &=&(\nabla _{X}\nabla _{Y}\phi )Z-(\nabla _{Y}\nabla _{X}\phi
)Z-(\nabla _{\left[ X,Y\right] }\phi )Z  \notag \\
&=&-g(R(X,Y)\xi ,\phi Z)\xi -\eta (Z)\phi R(X,Y)\xi  \notag \\
&&-g(\mathcal{A}Y,\phi Z)\mathcal{A}X+g(\mathcal{A}X,\phi Z)\mathcal{A}Y 
\notag \\
&&-g(\mathcal{A}X,Z)\phi \mathcal{A}Y+g(\mathcal{A}Y,Z)\phi \mathcal{A}X.
\label{rxy2}
\end{eqnarray}%
Taking into account (\ref{1}), we deduce%
\begin{equation*}
g(R(X,Y)\phi Z,\phi W)+g(R(X,Y)Z,W)=g(R(X,Y)Z,\xi )\eta (W)+g((R(X,Y)\phi
)Z,\phi W).
\end{equation*}%
Taking the inner product of (\ref{rxy2}) with $\phi W$, and using the above
equation, we get (\ref{RXYY}).
\end{proof}

\begin{proposition}
\label{b30}A \textit{quasi-para-Sasakian} manifold $(M^{2n+1},\phi ,\xi
,\eta ,g)~$satisfies followings%
\begin{eqnarray}
S^{\ast }(Y,Z)+S(Y,Z) &=&S(Y,\xi )\eta (Z)+g(\mathcal{A}Y,\phi Z)trace(\phi 
\mathcal{A})  \label{S1} \\
&&-g(\mathcal{A}Y,\mathcal{A}Z),  \notag \\
r^{\ast }+r &=&-tr^{2}(\phi \mathcal{A}),  \label{S2}
\end{eqnarray}%
where $S^{\ast }(X,Y)=\dsum\limits_{i=1}^{2n+1}\varepsilon
_{i}g(R(e_{i},X)\phi Y,\phi e_{i})$ denotes the *-Ricci curvature tensor and 
$r^{\ast }=\dsum\limits_{i=1}^{2n+1}\varepsilon _{i}S^{\ast }(e_{i},e_{i})$
denotes the *-scalar curvature of the $(M,\phi ,\xi ,\eta ,g),$ where $%
\left\{ e_{i}\right\} ,i\in \left\{ 1,...,2n+1\right\} $ be a local $\phi $%
-basis.
\end{proposition}

\begin{proof}
One can show that $\dsum\limits_{i=1}^{2n+1}g(\mathcal{A}e_{i},e_{i})=0$.
Using (\ref{p1}), (\ref{P2}) and (\ref{P4}), we get%
\begin{eqnarray}
\dsum\limits_{i=1}^{2n+1}g(\mathcal{A}e_{i},\phi Z)g(\mathcal{A}Y,\phi
e_{i}) &=&\dsum\limits_{i=1}^{2n+1}g(\mathcal{A}\phi Y,e_{i})g(\mathcal{A}%
\phi Z,e_{i})  \notag \\
&=&g(\mathcal{A}\phi Y,\mathcal{A}\phi Z)=-g(\mathcal{A}Y,\mathcal{A}Z).
\label{M1}
\end{eqnarray}%
Using the fact that $tr(\phi \mathcal{A})=\dsum\limits_{i=1}^{2n+1}%
\varepsilon _{i}g(\phi \mathcal{A}e_{i},e_{i})=-\dsum\limits_{i=1}^{2n+1}g(%
\mathcal{A}e_{i},\phi e_{i})$, (\ref{R 1.1}) and (\ref{M1}) after replacing $%
X,W$ by $e_{i}~$in (\ref{RXYY}) and taking summation over $i$, we find (\ref%
{S1}). For the proof of (\ref{S2}), after replacing $Y,Z$ by $e_{i}~$in (\ref%
{S1}) and taking the summation over $i$, and using (\ref{R1.3}), we obtain
the requested equation.
\end{proof}

\section{$D$-homothetic deformations}

Let $(M^{2n+1},\phi ,\xi ,\eta ,g)$ be an almost paracontact metric manifold
and $(\phi ,\xi ,\eta ,g)$ is an almost paracontact metric structure on $%
(M^{2n+1},\phi ,\xi ,\eta ,g)$. Tensor fields $\tilde{\phi},\tilde{\xi},%
\tilde{\eta}$ and $\tilde{g}$ defined as%
\begin{equation}
\tilde{\phi}=\phi ,\text{ \ \ }\tilde{\xi}=\frac{1}{\alpha }\xi ,\text{ }%
\tilde{\eta}=\alpha \eta ,\text{ \ }\tilde{g}=\beta g+(\alpha ^{2}-\beta
)\eta \otimes \eta ,  \label{deformatio}
\end{equation}%
where $\alpha \neq 0$ and $\beta >0$.

Thus, $(\tilde{\phi},\tilde{\xi},\tilde{\eta},\tilde{g})$ is also an almost
paracontact metric structure on $(M^{2n+1},\phi ,\xi ,\eta ,g).$

If the almost paracontact metric structures $(\phi ,\xi ,\eta ,g)$\ and $(%
\tilde{\phi},\tilde{\xi},\tilde{\eta},\tilde{g})$ are related with (\ref%
{deformatio}), then $(\tilde{\phi},\tilde{\xi},\tilde{\eta},\tilde{g})$ is
said to be $D$-homothetic to $(\phi ,\xi ,\eta ,g),$ namely, the almost
paracontact metric structure $(\tilde{\phi},\tilde{\xi},\tilde{\eta},\tilde{g%
})$ is obtained by a $D$\textit{-homothetic deformation} of the almost
paracontact metric structure $(\phi ,\xi ,\eta ,g)$. If $\alpha ^{2}=\beta $%
, then $D$-homothetic deformation will be called \textit{homothetic
deformation} \cite{TAN2}.

\begin{proposition}
\label{sasakian}If $(\phi ,\xi ,\eta ,g)$ is a \textit{quasi-para-Sasakian}
structure, then the structure $(\tilde{\phi},\tilde{\xi},\tilde{\eta},\tilde{%
g})$ is also \textit{quasi-para-Sasakian}. If $(\phi ,\xi ,\eta ,g)$ is 
\textit{para-Sasakian, then }$(\tilde{\phi},\tilde{\xi},\tilde{\eta},\tilde{g%
})~$is \textit{para-Sasakian if and only if }$\alpha =\beta $.
\end{proposition}

\begin{proof}
By virtue of Definition \ref{mert} and (\ref{deformatio}), we obtain the
assertion.
\end{proof}

\begin{lemma}
\label{nambla}Let $(\tilde{\phi},\tilde{\xi},\tilde{\eta},\tilde{g})$ be a 
\textit{quasi-para-Sasakian} structure obtained from $(\phi ,\xi ,\eta ,g)$
by a $D$-homothetic deformation. Then we have the following relation between
the Levi-Civita connections $\tilde{\nabla}$and $\nabla .$%
\begin{equation}
\tilde{\nabla}_{X}Y=\nabla _{X}Y+\left( \frac{\alpha ^{2}}{\beta }-1\right)
(\eta (Y)\mathcal{A}X+\eta (X)\mathcal{A}Y).  \label{i00}
\end{equation}
\end{lemma}

\begin{proof}
By Koszul formula we have 
\begin{eqnarray*}
2\tilde{g}(\tilde{\nabla}_{X}Y,Z) &=&X\tilde{g}(Y,Z)+Y\tilde{g}(X,Z)-Z\tilde{%
g}(X,Y) \\
&&+\tilde{g}(\left[ X,Y\right] ,Z)+\tilde{g}(\left[ Z,X\right] ,Y)+\tilde{g}(%
\left[ Z,Y\right] ,X),
\end{eqnarray*}%
for any vector fields $X,Y,Z$. Using $\tilde{g}=\beta g+(\alpha ^{2}-\beta
)\eta \otimes \eta $ and (\ref{p1}) in the last equation, we obtain%
\begin{equation*}
2\tilde{g}(\tilde{\nabla}_{X}Y,Z)=2\beta g(\nabla _{X}Y,Z)+2(\alpha
^{2}-\beta )\left[ \eta (\nabla _{X}Y)\eta (Z)+g(Z,\mathcal{A}X)\eta (Y)+g(Z,%
\mathcal{A}Y)\eta (X)\right] .
\end{equation*}%
Since $g(Z,\mathcal{A}Y)=g(\mathcal{A}Y,Z),$ we get%
\begin{eqnarray}
2\tilde{g}(\tilde{\nabla}_{X}Y,Z) &=&2\beta g(\nabla _{X}Y,Z)  \label{i} \\
&&+2(\alpha ^{2}-\beta )\left[ \eta (X)g(\mathcal{A}Y,Z)+\eta (Y)g(\mathcal{A%
}X,Z)+\eta (Z)\eta (\nabla _{X}Y)\right] .  \notag
\end{eqnarray}%
Moreover, $\tilde{g}(\tilde{\nabla}_{X}Y,Z)$ is equal to 
\begin{equation}
\beta g(\tilde{\nabla}_{X}Y,Z)+(\alpha ^{2}-\beta )\eta (\tilde{\nabla}%
_{X}Y)\eta (Z).  \label{i1}
\end{equation}%
Substituting (\ref{i1}) in (\ref{i}), we obtain%
\begin{eqnarray}
&&\beta g(\tilde{\nabla}_{X}Y,Z)+(\alpha ^{2}-\beta )\eta (\tilde{\nabla}%
_{X}Y)\eta (Z)  \notag \\
&=&\beta g(\nabla _{X}Y,Z)+(\alpha ^{2}-\beta )\left[ \eta (X)g(\mathcal{A}%
Y,Z)+\eta (Y)g(\mathcal{A}X,Z)+\eta (Z)\eta (\nabla _{X}Y)\right] .
\label{i22}
\end{eqnarray}%
Setting $Z=\xi $ in (\ref{i22}) and using (\ref{p1}), we get%
\begin{equation}
\eta (\tilde{\nabla}_{X}Y)=\eta (\nabla _{X}Y).  \label{i33}
\end{equation}%
(\ref{i00}) is a direct consequence of (\ref{i22}) and (\ref{i33}).
\end{proof}

\begin{proposition}
\label{relations}Let $(M^{2n+1},\phi ,\xi ,\eta ,g)$ and $(\tilde{M}^{2n+1},%
\tilde{\phi},\tilde{\xi},\tilde{\eta},\tilde{g})$ are locally $D$-homothetic 
\textit{quasi-para-Sasakian} manifolds. Then following identities hold:%
\begin{equation}
\overset{\sim }{\mathcal{A}}X=\frac{\alpha }{\beta }\mathcal{A}X,  \label{i5}
\end{equation}%
\begin{equation}
\tilde{g}(\mathcal{A}X,Y)=\alpha g(\mathcal{A}X,Y),  \label{i6}
\end{equation}%
\begin{eqnarray}
\tilde{R}(X,Y)Z &=&R(X,Y)Z  \notag \\
&&-\left( \frac{\alpha ^{2}}{\beta }-1\right) \left\{ g(\mathcal{A}Y,Z)%
\mathcal{A}X-g(\mathcal{A}X,Z)\mathcal{A}Y-2g(\mathcal{A}X,Y)\mathcal{A}%
Z\right\}  \notag \\
&&+\left( \frac{\alpha ^{2}}{\beta }-1\right) ^{2}\left\{ \eta (X)\eta (Z)%
\mathcal{A}^{2}Y-\eta (Y)\eta (Z)\mathcal{A}^{2}X\right\}  \notag \\
&&+\left( \frac{\alpha ^{2}}{\beta }-1\right) \left\{ \eta (X)R(\xi
,Y)Z+\eta (Y)R(X,\xi )Z+\eta (Z)R(X,Y)\xi \right\} ,  \label{i777}
\end{eqnarray}%
for any vector fields $X,Y,Z.$
\end{proposition}

\begin{proof}
After setting $Y=\xi $ in (\ref{i00}), if we use $\tilde{\xi}=\frac{1}{%
\alpha }\xi $ and (\ref{P6}), we get (\ref{i5}). Using (\ref{deformatio})
and (\ref{i00}), after some calculations one can obtain (\ref{i6}). From the
curvature formula 
\begin{equation*}
\tilde{R}(X,Y)Z=[\tilde{\nabla}_{X},\tilde{\nabla}_{Y}]Z-\tilde{\nabla}%
_{[X,Y]}Z,
\end{equation*}%
Eq. (\ref{p1}), (\ref{R1}) and (\ref{i00}), after a straightforward
computation one can get (\ref{i777}).
\end{proof}

Since the proof of the following proposition is quite similar to Proposition
4.4 of \cite{OL1}, so we don't give the proof of it.

\begin{proposition}
\label{equivalent}Let $(\phi ,\xi ,\eta ,g)$ be a \textit{%
quasi-para-Sasakian structure.} Then the following assertions are equivalent
to each other:
\end{proposition}

$i)$ $(\phi ,\xi ,\eta ,g)$ \textit{can be obtained by a }$D$\textit{%
-homothetic deformation of a} \textit{para-Sasakian structure,}

$ii)$ $(\phi ,\xi ,\eta ,g)$ \textit{can be obtained by a homothetic
deformation of a} \textit{para-Sasakian structure,}

$iii)$ $\mathcal{A}X=\lambda \phi X$, for $\lambda $ $=$ const.$\neq $ 0.

\section{Quasi-Para-Sasakian manifolds of constant curvature}

\begin{theorem}
\label{teo.}Let $(M^{2n+1},\phi ,\xi ,\eta ,g)$ be a \textit{%
quasi-para-Sasakian manifold} of constant curvature $K$. Then $K\leq 0.$
Furthermore,
\end{theorem}

\textit{\ \ \ \ \ \ \ }$\bullet $\textit{If }$K=0,$ \textit{the manifold is
paracosymplectic,}

\textit{\ \ \ \ \ \ \ }$\bullet $\textit{If }$K<0,$\textit{\ the structure }$%
(\phi ,\xi ,\eta ,g)$\textit{\ is obtained by a homothetic deformation of a
para-Sasakian structure on }$M^{2n+1}.$

\begin{proof}
One can see that $K\leq 0$ from Lemma \ref{2-form}. If $K=0,$\ by (\ref{R1.3}%
), we get $\mathcal{A}=0$. Hence, from (\ref{P5}), we have $\nabla \phi =0$.
This means the manifold is paracosymplectic. Assume that $K<0$. The claim
follows from Proposition \ref{equivalent}. So, we should obtain $\mathcal{A}%
X=\lambda \phi X$, for $\alpha $ $=$ const.$\neq $ 0. After straightforward
calculations we have $r=2n(2n+1)K$ and $r^{\ast }=-2nK$. If we use these
equations in (\ref{S2}), we obtain 
\begin{equation}
tr(\phi \mathcal{A})=2n\lambda ,\text{ so }K=-\lambda ^{2}.  \label{irem}
\end{equation}%
By direct calculations, we get%
\begin{eqnarray}
S(Y,Z) &=&2nKg(Y,Z)\text{ and }  \label{irem2} \\
S^{\ast }(Y,Z) &=&K(-g(Y,Z)+\eta (Y)\eta (Z))\overset{(\ref{R 1.1})}{=}-g(%
\mathcal{A}Y,\mathcal{A}Z).  \notag
\end{eqnarray}%
Making use of (\ref{irem}) and (\ref{irem2}) in (\ref{S1}), we deduce%
\begin{equation}
g(\mathcal{A}Y,\phi Z)=-\lambda (g(Y,Z)-\eta (Y)\eta (Z)).  \label{irem3}
\end{equation}%
Putting $\phi Y$ for $Y$ in (\ref{irem3}) and using (\ref{P3}), we have $%
\mathcal{A}Y=\lambda \phi Y$. This completes the proof.
\end{proof}

\section{Example}

Now, we will give an example of $3$\textit{-}dimensional\textit{\ }proper
quasi-para-Sasakian manifold.

\begin{example}
\label{ex2} We consider the $3$-dimensional manifold%
\begin{equation*}
M^{3}=\{(x,y,z)\in 
\mathbb{R}
^{3},z\neq 0\}
\end{equation*}%
and the vector fields%
\begin{equation*}
\phi e_{2}=e_{1}=4y\text{ }\frac{\partial }{\partial x}+z\frac{\partial }{%
\partial z},\text{ \ \ }\phi e_{1}=e_{2}=\frac{\partial }{\partial y},\text{
\ \ }\xi =e_{3}=\frac{\partial }{\partial x}.
\end{equation*}%
The $1$-form $\eta =dx-\frac{4y}{z}dz$ defines an almost paracontact
structure on $M$ with characteristic vector field $\xi =\frac{\partial }{%
\partial x}$. Let $g$, $\phi $ be the semi-Riemannian metric ($%
g(e_{1},e_{1})=-g(e_{2},e_{2})=g(\xi ,\xi )=1)$ and the $(1,1)$-tensor field
respectively given by 
\begin{eqnarray*}
g &=&\left( 
\begin{array}{ccc}
1 & 0 & -\frac{2y}{z} \\ 
0 & -1 & 0 \\ 
-\frac{2y}{z} & 0 & \frac{1+28y^{2}}{z^{2}}%
\end{array}%
\right) ,\text{ } \\
\phi &=&\left( 
\begin{array}{ccc}
0 & 4y & 0 \\ 
0 & 0 & \frac{1}{z} \\ 
0 & z & 0%
\end{array}%
\right) ,
\end{eqnarray*}%
with respect to the basis $\frac{\partial }{\partial x},\frac{\partial }{%
\partial y},\frac{\partial }{\partial z}$.
\end{example}

\textit{Using }$\nabla _{X}\xi =\beta \phi X$\textit{\ (see \cite{Welyczko1}%
) we have}%
\begin{equation*}
\begin{array}{ccc}
\nabla _{e_{1}}e_{1}=0,~ & \nabla _{e_{2}}e_{1}=-2\xi , & ~~~\nabla _{\xi
}e_{1}=2e_{2}, \\ 
\nabla _{e_{1}}e_{2}=2\xi , & \nabla _{e_{2}}e_{2}=0,~ & \nabla _{\xi
}e_{2}=2e_{1}, \\ 
\nabla _{e_{1}}\xi =2e_{2}, & ~~\nabla _{e_{2}}\xi =2e_{1}, & \nabla _{\xi
}\xi =0.%
\end{array}%
\end{equation*}%
\textit{Hence the manifold is a }$3$\textit{-dimensional quasi-para-Sasakian
manifold with }$\beta $\textit{\ is constant function. Using the above
equations, we obtain }%
\begin{equation}
\begin{array}{ccc}
R(e_{1},e_{2})\xi =0, & R(e_{2},\xi )\xi =-4e_{2}, & R(e_{1},\xi )\xi
=-4e_{1}, \\ 
R(e_{1},e_{2})e_{2}=-12e_{1}, & R(e_{2},\xi )e_{2}=-4\xi , & R(e_{1},\xi
)e_{2}=0, \\ 
R(e_{1},e_{2})e_{1}=-12e_{2}, & R(e_{2},\xi )e_{1}=0, & R(e_{1},\xi
)e_{1}=4\xi .%
\end{array}
\label{r1}
\end{equation}%
\textit{Using (\ref{r1}), we have constant scalar curvature as follows, }$%
r=S(e_{1},e_{1})-S(e_{2},e_{2})+S(\xi ,\xi )=8.$\textit{\ We want to remark
that this example is neither the paracosymplectic manifold nor the
para-Sasakian manifold example.}

\end{document}